\documentclass[11pt]{amsart}
\usepackage[utf8]{inputenc}

\usepackage{amsmath}
\usepackage{amsthm}
\usepackage{accents}
\usepackage{amssymb}
\usepackage{graphicx}
\usepackage{color}
\usepackage{enumerate}

\newcommand{\del}{\partial}
\newcommand{\CL}{\mathcal{L}}

\newcommand{\R}{\mathbb{R}}
\renewcommand{\S}{\mathbb{S}}

\newcommand{\Scal}{\operatorname{R}}
\newcommand{\C}{\operatorname{\mathcal{C}}}

\newcommand{\Ric}{\operatorname{Rc}}

\renewcommand{\div}{\operatorname{div}}

\newcommand {\mc} {H}

\renewcommand {\H} {\mathcal{H}}
\newcommand {\spt} {\operatorname{spt}}

\newtheorem{theorem}{Theorem}[section]

\newtheorem{lemma}[theorem]{Lemma}
\newtheorem{definition}[theorem]{Definition}
\newtheorem{corollary}[theorem]{Corollary}
\newtheorem{proposition}[theorem]{Proposition}

\title[Stable CMC surfaces]{On large volume preserving stable CMC surfaces in initial data sets}

\author{Michael Eichmair \and Jan Metzger}
\address{Michael Eichmair, Departement Mathematik, ETHZ, 8092 Z\"urich, Switzerland}
\email{michael.eichmair@math.ethz.ch}

\address{Jan Metzger, Universit\"at Potsdam, Institut f\"ur Mathematik, Am
Neuen Palais 10, 14469 Potsdam, Germany}
\email{jan.metzger@uni-potsdam.de}

\thanks{Michael Eichmair gratefully acknowledges the support of the NSF grant DMS-0906038.}

\begin{document}

\begin{abstract} Let $(M, g)$ be a complete $3$-dimensional
  asymptotically flat manifold with \emph{everywhere positive} scalar
  curvature. We prove that, given a compact subset $K \subset M$, all
  volume preserving stable constant mean curvature surfaces of
  sufficiently large area will avoid $K$. This complements the results
  of G.~Huisken and S.-T.~Yau \cite{Huisken-Yau:1996} and of J.~Qing and
  G.~Tian \cite{Qing-Tian:2007} on the uniqueness of large volume
  preserving stable constant mean curvature spheres in initial data
  sets that are asymptotically close to Schwarzschild with mass $m >
  0$. The analysis in \cite{Huisken-Yau:1996} and
  \cite{Qing-Tian:2007} takes place in the asymptotic regime of
  $M$. Here we adapt ideas from the minimal surface proof of the
  positive mass theorem \cite{Schoen-Yau:1979-pmt1} by R.~Schoen and
  S.-T.~Yau and develop geometric properties of volume preserving
  stable constant mean curvature surfaces to handle surfaces that run
  through the part of $M$ that is far from Euclidean.
\end{abstract}

\maketitle

\section {Introduction}

A classical result in the calculus of variations is that the
isoperimetric regions of $\R^n$ are precisely the round balls. A surface which is critical for the associated variational problem has constant mean curvature. In 1951, H.~Hopf
proved that immersed two-spheres of constant mean curvature in $\R^3$
are necessarily round, and then, in 1958, A. D. Alexandrov showed that so are compact embedded constant mean curvature hypersurfaces in
$\R^n$.  Note that
translations preserve such surfaces and fully
account for their non-uniqueness, once we fix their area. We recommend the wonderful article \cite{Osserman:1978} by
R. Osserman for a survey on the isoperimetric problem.

In 1988, G. Huisken and S.-T. Yau made the crucial observation that in
a large and physically important class of asymptotically flat manifolds, this ambiguity disappears: they proved that certain large, volume preserving stable
constant mean curvature spheres exist and are \emph{unique} (given their area) within a large class of surfaces, including all nearby ones. Their insight has started a long line of remarkable
research linking the geometric properties of such surfaces to the
physical properties of the asymptotically flat manifold when viewed as
(time-symmetric) initial data for the Einstein equations.

The present paper complements the existing results, which we discuss
in more detail below, on the uniqueness of large, volume preserving
stable constant mean curvature surfaces in initial data sets. We set the stage with the relevant definitions.

\begin{definition} \label{def:initial_data_sets} An initial data set $(M, g)$ is a connected complete Riemannian $3$-manifold, possibly with compact boundary, such that there exists a bounded open set $U \subset M$ so that $M \setminus U \cong_x \R^3 \setminus B(0, \frac{1}{2})$ and such that in the coordinates induced by $x = (x_1, x_2, x_3)$ we have that $$ r |g_{ij} - \delta_{ij}| + r^2 |\partial_k g_{ij}| + r^3 |\partial^2_{k l} g_{ij}| \leq C$$ where $r:= \sqrt {x_1^2 + x_2^2 + x_3^2}$. The boundary $\partial M$ of $M$ -- if non-empty -- is a minimal surface, and we assume that there are no other compact minimal surfaces in $M$. The boundary of $M$ is called the horizon of $(M, g)$. Given $m \in [0, \infty)$ and an integer $k\geq0$ we say that an initial data set is $\C^k$-asymptotic to Schwarzschild of mass $m$ if 
\begin{eqnarray} \label{eqn:decaygm}
\sum_{l=0}^k  r^{2 + l}|\partial^l(g - g_m)_{ij}| \leq C
\end{eqnarray} where $(g_m)_{ij} = (1 + \frac{m}{2|x|})^4 \delta_{ij}$. 
\end{definition}

We do not assume here that an initial data set has non-negative scalar
curvature. For convenience, we extend $r$ as a smooth regular function to the entire initial data set $(M, g)$ such that $r(U) \subset [0, 1)$. Note that if $(M, g)$ is an initial data set that is
$\C^1$-asymptotic to Schwarzschild with mass $m$, then $m$ equals the
ADM-mass of $(M, g)$. 

\begin{definition} [cf. \cite{Barbosa-DoCarmo-Eschenburg}] Let $(M,
  g)$ be a Riemannian $3$-manifold and let $\Sigma \subset M$ be a
  complete embedded two-sided boundaryless surface. Let $\nu$ be a smooth
  unit normal vector field of $\Sigma$. The mean curvature
  $\mc$ of $\Sigma$ (with respect to $\nu$) is defined as the
  tangential divergence of $\nu$. We say that $\Sigma$ has constant
  mean curvature if $\mc$ is constant along $\Sigma$. We say that a
  constant mean curvature surface $\Sigma \subset M$ is volume
  preserving stable if
  \begin{multline*}
    \int_\Sigma (|h|^2 + \Ric(\nu, \nu)) u^2 d
    \H^2_g \leq \int_\Sigma |\bar \nabla u|^2 d \H^2_g
    \\ \text { for every
    } u \in \C^1_c(\Sigma) \text { with } \int_\Sigma u d \H^2_g= 0.
  \end{multline*}
  Here, $\Ric$ is the ambient Ricci tensor, $\bar \nabla$ is the
  tangential gradient along $\Sigma$, and $h$ denotes the second
  fundamental form tensor of $\Sigma$. We say that a constant mean
  curvature surface $\Sigma \subset M$ is strongly stable if $$
  \int_\Sigma (|h|^2 + \Ric(\nu, \nu)) u^2 d\H^2_g \leq \int_\Sigma
  |\bar \nabla u|^2d \H^2_g \text { for every } u \in
  \C^1_c(\Sigma).$$
\end{definition}

We caution the reader that in \cite{Huisken-Yau:1996, Qing-Tian:2007}, volume preserving stable constant mean curvature surfaces are referred to as (weakly) stable. The distinction between volume preserving stable and strongly stable constant mean curvature surfaces is important in this paper.  

The notions of constant mean curvature and volume preserving stability extend to isometrically immersed surfaces in $(M, g)$ that are two-sided in the sense that the surface has a global ``unit normal" in the pull-back of the tangent bundle of $M$ to the surface. Note that non-zero constant mean curvature is a local notion for a connected surface, and two-sidedness is a consequence.   

Volume preserving stable constant mean curvature surfaces are precisely the stable critical points for volume-constrained area minimization in $(M, g)$, i.e., the isoperimetric problem \cite{Barbosa-DoCarmo-Eschenburg}. Arguably, they are the surfaces that generalize the variational properties of the outermost minimal surface in $(M, g)$ -- the (apparent) horizon -- most naturally. In \cite{Christodoulou-Yau:1988} it was shown that in initial data sets with non-negative scalar curvature, a connected volume preserving stable constant mean curvature \emph{sphere} has non-negative Hawking mass. In their seminal paper \cite{Huisken-Yau:1996} (see also the announcement in \cite[p. 14]{Christodoulou-Yau:1988}) G. Huisken and S.-T. Yau showed that the exterior of a large compact subset of an initial data set $(M, g)$ that is $\C^4$-asymptotic to Schwarzschild with mass $m>0$ is foliated by volume preserving stable constant mean curvature spheres. These spheres become rounder as they diverge to infinity, and their centers of mass with respect to the Euclidean coordinate system at infinity converge to a unique point in the limit: the Huisken-Yau ``geometric center of mass" of the initial data set. The existence of a constant mean curvature foliation has also been shown by R. Ye \cite{Ye:1996} using a different approach. Importantly, G. Huisken and S.-T. Yau were able to show that the volume preserving stable constant mean curvature spheres are unique within a large class of surfaces:

\begin{theorem} [\cite{Huisken-Yau:1996}] \label{thm:Huisken-Yau} Let $(M, g)$ be $\mathcal{C}^4$-asymptotic to Schwarzschild with mass $m > 0$ and let $q \in (\frac{1}{2},1]$ be given. There exists $\mc_0> 0$ depending only on $m$ and $C$ as in
\eqref{eqn:decaygm} and $q$  such that for every $\mc \in (0, \mc_0)$ there is a unique volume preserving stable constant mean curvature sphere of mean curvature $\mc$ that contains the ball $B_{\mc^{-q}}$. These constant mean curvature spheres foliate the exterior of a compact subset of $M$. 
\end{theorem}

J. Qing and G. Tian strengthened the uniqueness result of G. Huisken and S.-T.~Yau as follows:

\begin{theorem} [\cite{Qing-Tian:2007}] \label{thm:Qing-Tian} Let $(M, g)$ be $\mathcal{C}^4$-asymptotic to Schwarzschild with mass $m > 0$. There exist $r_0 \geq 1$ and $\mc_0 >0$ depending only on $m$ and $C$ as in \eqref{eqn:decaygm} such that for every $\mc \in (0, \mc_0)$ there exists a unique volume preserving stable constant mean curvature sphere of mean curvature $\mc$ in $(M,g)$ that contains the ball $B_{r_0}$. \end{theorem}

These theorems have been partially extended to more general asymptotic expansions of the metric at infinity in \cite[Theorem 3]{Huang:2010} and \cite[Theorem 1.5]{Ma:2011}. 

In this paper, we set out to complete the description of large, volume preserving stable constant mean curvature surfaces in initial data sets with positive scalar curvature, complementing the remarkable works of \cite{Huisken-Yau:1996} and \cite{Qing-Tian:2007}. Our main results here are as follows: 

\begin{theorem}\label{thm:volume_stable_center} Let $(M, g)$ be an initial data set with non-negative scalar curvature. Assume that the scalar curvature is positive in a neighborhood of a non-empty compact subset $K \subset M$. For every $\Theta > 0$ there exists a constant $A = A(M, g, \Theta, K)>0$ such that there are no connected closed volume preserving stable constant mean curvature surfaces $\Sigma \subset M$ with $\H^2_g(\Sigma \cap B_\sigma) \leq \Theta \sigma^2$ for all $\sigma \geq 1$ and with $\H^2_g(\Sigma) \geq A$ such that $\Sigma \cap K \neq \emptyset$.  
\end{theorem}

We show in Corollary \ref{cor:quadratic_area_growth} that the assumption of quadratic area growth is not nearly as stringent as it might appear. We can drop this assumption altogether if we assume that the scalar curvature of $(M, g)$ is \emph{everywhere} positive: 

\begin{theorem} \label{thm:volume_stable_center_positive} Let $(M, g)$ be an initial data set with everywhere positive scalar curvature. Given $r \geq 1$ there exists a constant $A_r > 0$ so that every connected closed volume preserving stable constant mean curvature surface $\Sigma \subset M$ with area $\H^2_g(\Sigma) \geq A_r$ is disjoint from the ball $B_r$.    
\end{theorem}
The existence of a foliation of the asymptotic regime through constant
mean curvature surfaces in initial data sets that are only
$\mathcal{C}^2$-asymptotic to Schwarzschild with mass $m>0$ has been
established in \cite{Metzger:2007}. The assertions about the
uniqueness of volume preserving stable constant mean curvature
surfaces also extends to such initial data sets, see \cite{Eichmair-Metzger:2012hdiso}:

\begin{theorem} \label{thm:Huisken-Yau-Qing-Tian_improved} Theorems
  \ref{thm:Huisken-Yau} and \ref{thm:Qing-Tian} hold for initial data
  sets that are $\mathcal{C}^2$-asymptotic to Schwarzschild with mass
  $m>0$. The uniqueness statements apply to volume preserving stable
  constant mean curvature surfaces of arbitrary genus (not only
  spheres).
\end{theorem}
In conjunction with Theorem
\ref{thm:volume_stable_center_positive} this leads to the following
complete description of large, volume preserving stable constant mean
curvature surfaces in initial data sets, which complements the
existing results on the ``global uniqueness problem for stable
constant mean curvature surfaces" (cf. \cite[bottom of
p. 301]{Huisken-Yau:1996} and \cite[footnote on
p. 1092]{Qing-Tian:2007}):

\begin{theorem} \label{thm:global_uniqueness} Let $(M, g)$ be an initial data set that is $\C^2$-asymptotic to Schwarz\-schild with mass $m>0$ and whose scalar curvature is everywhere positive. For a fixed point $p \in M$ there exists a constant $A > 0$ such that every connected closed volume preserving stable constant mean curvature surface in $M$ that is (together with the horizon, if it is non-empty) the boundary of a compact set containing $p$ and whose area is at least $A$ is uniquely determined by its area. In fact, these surfaces foliate the exterior of some bounded region in $M$.  
\end{theorem}

The analysis of \cite{Huisken-Yau:1996} and \cite{Qing-Tian:2007} takes place in the asymptotic regime of the initial data set. That $m>0$ is crucial at many points in their arguments. While the particular expansion of the metric at infinity is immaterial in our analysis, we depend on our assumption that the scalar curvature is non-negative. At the heart of the proofs of our main theorems is a fundamental mechanism discovered by R. Schoen and S.-T. Yau in their celebrated proof  of the positive mass theorem \cite{Schoen-Yau:1979-pmt1}: positive ambient scalar curvature is not compatible with the existence of (certain) strongly stable minimal surfaces. In order to exploit this mechanism, we need several additional insights on the behavior of large, volume preserving stable constant mean curvature surfaces in initial data sets, including refined curvature estimates, bounds on their area growth, and the observation that they limit to strongly stable minimal surfaces. 

The following example in Schwarzschild shows that we can only expect
to prove a uniqueness result for \emph{large} volume preserving stable
constant mean curvature surfaces. 

\noindent {\bf Example} (Round spheres in Schwarzschild): Let $(g_m)_{ij} = (1 + \frac{m}{2r})^4 \delta_{ij} = \phi_m^4 \delta_{ij}$ be the Schwarzschild metric of mass $m>0$ on $\R^3 \setminus \{0\}$ and let $S_r := \{x \in \R^3 : |x| = r\}$ denote the centered spheres for $r>0$. Recall that $g_m$ is a complete scalar flat metric, and that $x \to \left(\frac{m}{2}\right)^2\frac{x}{|x|^2}$ is a reflection isometry about the horizon $S_{\frac{m}{2}}$. The spheres $S_r$ are umbilic and have constant mean curvature  $\frac{1 - m/(2r)}{\phi_m^3}\frac{2}{r}$. Note that the mean curvature is increasing for $\frac{m}{2} \leq r \leq \frac{m (2 + \sqrt{3})}{2}$ and decreasing for $r \geq  \frac{m (2 + \sqrt{3})}{2}$. The sphere $S_{ \frac{m (2 + \sqrt{3})}{2}}$  of largest mean curvature is called the photonsphere. The stability operator $L_{S_r} = - \bar \Delta - (|h|^2 + \Ric(\nu, \nu))$, where $\bar \Delta$ is the (negative definite) Laplacian for the induced metric on $S_r$, is easily computed to be $- \phi^{-4} r^{-2} \Delta_{\mathbb{S}^2} + \frac{- 4r^2 + 8rm - m^2}{2r^4 \phi^6}$, where $\Delta_{\mathbb{S}^2}$ is the Laplacian on the round unit sphere. The eigenfunctions of this operator are those of $\Delta_{\S^2}$. It follows that $\lambda_0 = \frac{- 4r^2 + 8rm - m^2}{2r^4 \phi^6}$ (with the constants spanning the eigenspace), and that $\lambda_1 = \lambda_2 = \lambda_3 =  \frac{6m}{r^3 \phi^6}$. Hence $S_r$ is strongly stable for $\frac{m}{2} \leq r <  \frac{m (2 + \sqrt{3})}{2} $ and volume preserving stable for $r \geq  \frac{m (2 + \sqrt{3})}{2}$.   

The preceding example does not have positive scalar curvature and hence does not quite satisfy the conditions of Theorem \ref{thm:volume_stable_center_positive}. We can choose $\epsilon, m >0$ such that the metrics $(1 + \frac{m }{2r} - \frac{\epsilon}{r^2})^4 \delta_{ij}$ on $\{x \in \R^3 : |x| > \underline {r} \}$ have positive scalar curvature, where $\underline r$ is the largest zero of $1 + \frac{m }{2r} - \frac{\epsilon}{r^2}$, and such that for some $r_h > \underline r$ the coordinate sphere $S_{r_h}$ is a minimal surface. The properties of the coordinate spheres $S_r$ for $r \geq r_h$ with respect to this metric are similar to those in the Schwarzschild example above.

S. Brendle has shown recently \cite{Brendle:2011} that the only closed constant mean curvature surfaces in Schwarzschild that do not intersect the horizon $S_{\frac{m}{2}}$ are the umbilic spheres $S_r$. The first author and S. Brendle \cite{Brendle-Eichmair:2011} have classified the``null-homologous" isoperimetric surfaces in  Schwarzschild and showed that there exist small, null-homologous, volume preserving stable constant mean curvature surfaces in Schwarzschild that intersect the horizon.  

H. Bray observed in his thesis \cite{Bray:1998} that the Hawking mass is monotone along any area-increasing foliation through connected volume preserving stable constant mean curvature spheres in initial data sets with non-negative scalar curvature. He used this to prove a special case of the Riemannian Penrose inequality using isoperimetric surface techniques. He conjectured that the volume preserving stable constant mean curvature spheres found by G. Huisken and S.-T. Yau are isoperimetric and proved this for the exact Schwarzschild metric. In \cite{Eichmair-Metzger:2010} we confirmed H. Bray's conjecture and in fact proved something stronger: 

\begin{theorem} [\cite{Eichmair-Metzger:2010}] Let $(M, g)$ be an initial data set that is $\C^0$-asymptotic to Schwarzschild with mass $m>0$ in the sense of Definition \ref{def:initial_data_sets}. There exists $V_0 >0$ such that for every $V \geq V_0$ the infimum in 
\begin{eqnarray} \label{eqn:isoperimetric_profile}
  A_g(V) := \inf \{ \H^2_g(\del^*\Omega) : \Omega \subset M \text { is Borel,} \\ \text{contains the horizon, has finite perimeter, and } \CL^3_g(\Omega) = V \} \nonumber
\end{eqnarray}
is achieved by a smooth isoperimetric region $\Omega \subset M$ containing the horizon and of volume $V$. The boundary  $\partial \Omega$ of every minimizer $\Omega$ is close to a centered coordinate sphere.   
\end{theorem}

\subsection*{Structure of this paper} In Section \ref{sec:curvature_estimates} we derive useful curvature estimates \emph{with decay} for large, volume preserving stable constant mean curvature surfaces by refining an argument of R. Ye. These curvature estimates lead to a useful monotonicity formula ``at infinity" for non-compact minimal surfaces of quadratic area growth that implies that their blow-down is conical, as in W. Fleming's proof of the Bernstein theorem in $\R^3$. This will be important in Section \ref{sec:PMT} where we refine and extend the scope of an argument of R. Schoen and S.-T. Yau in their minimal surface proof of the positive mass theorem to apply to certain limits of volume preserving stable constant mean curvature surfaces. We collect some insights on the structure of large, volume preserving stable constant mean curvature surfaces -- in particular bounds on the number of their components -- in Section \ref{sec:number_of_components}. In Appendix \ref{sec:integral_decay_estimates} we collect (and slightly extend) some well-known estimates on integrals of polynomially decaying quantities over surfaces with bounded bending energy. In Appendix \ref{sec:Willmore_Appendix} we explain why the area growth of a surface of bounded bending energy is bounded, essentially by its area in a fixed compact set. To obtain such an initial area bound for the volume preserving stable constant mean curvature surfaces considered in this paper we analyze the proof of an estimate on the Hawking mass of such surfaces due to D. Christodoulou and S.-T. Yau (cited here as Lemma \ref{lem:Christodoulou-Yau}). This is where the assumption that the scalar curvature is positive enters crucially.    

In this paper we will distinguish conscientiously between immersed and properly immersed surfaces. A surface $\Sigma \subset M$, unless otherwise specified, is complete with no boundary, properly embedded, orientable, and two-sided.

\subsection*{Acknowledgements} We would like to thank Gerhard Huisken, Bennett Palmer, Jie Qing, Antonio Ros, Richard Schoen, Gang Tian, Shing-Tung Yau, and Rugang Ye for helpful discussions and their kind support. Michael Eichmair would like to also thank Johannes Nepomuk Gro\ss ruck, Thomas Kern, and Stefan Schmidinger for their great hospitality during times when important parts of this paper were written, and Otis Chodosh for carefully proofreading the final version of this paper.

%%%%%%%%%%%%%%%%%%%%%%%%%%%%%%%%%
%%%%%%%%%%%%%%%%%%%%%%%%%%%%%%%%%
%%%%%%%%%%%%%%%%%%%%%%%%%%%%%%%%%

\section{Curvature estimates through blow up} \label{sec:curvature_estimates}

In this section we discuss curvature estimates for volume preserving
stable constant mean curvature surfaces in homogeneously regular
$3$-manifolds. The arguments described here are by blow up and use a
characterization of volume preserving stable constant mean curvature immersions in Euclidean
space from  \cite{Barbosa-DoCarmo:1984}, \cite{Palmer:1986},
\cite{DaSilveira:1987}, \cite{Lopez-Ros:1989}: they are either spheres
or planes. The earliest reference containing this general line of
reasoning that we have been able to find is \cite[Theorem
7]{Ye:1996}. Variations of this argument (``local curvature estimates
via blow up and a Bernstein type theorem") in other related contexts
appear in \cite[Theorem 2.2]{Johnson-Morgan:2000} and \cite[Theorem
18]{Ros:2005} for isoperimetric surfaces, and in
\cite{Rosenberg-Souam-Toubiana:2010} for strongly stable constant mean
curvature surfaces. In \cite{Simon:1976}, curvature estimates for
minimizing boundaries in ambient dimensions $n \leq 7$ (and for graphs
when $n \leq 8$) have been obtained by similar reasoning. In
\cite{Rosenberg-Souam-Toubiana:2010} much care is applied to derive
curvature estimates that are essentially independent of ambient
geometric quantities. Their argument to show independence of the
injectivity radius -- by passing to the universal cover of the ambient
manifold -- does not carry over to our context: strong
stability lifts to the cover (cf.  \cite[Corollary
4]{Fischer-Colbrie-Schoen:1980} and \cite[Theorem
3.1]{Colding-Minicozzi:2002}), volume preserving
stability does not. (For example, there exist non-flat volume preserving
stable minimal surfaces in $3$-dimensional tori,
cf. \cite{Ross:1992}.) The iteration method of
\cite{Schoen-Simon-Yau:1975} has been adapted by
\cite{Huisken-Yau:1996} to derive curvature estimates for spherical
volume preserving stable constant mean curvature surfaces that lie far
out in the Euclidean end of an initial data set.

First we recall the beautiful characterization of volume preserving
stable constant mean curvature immersions that lies at the heart of these curvature
estimates, and is proven as \cite[Theorem 1.3]{Barbosa-DoCarmo:1984}
(when the immersion is compact), as \cite[Theorem 3.1]{Palmer:1986} (when the
immersion has non-zero mean curvature), and as \cite[Theorem
1.3]{DaSilveira:1987} and \cite[Theorem 5]{Lopez-Ros:1989} (in the
full generality as stated below):

\begin{lemma}[\protect{\cite{Barbosa-DoCarmo:1984, Palmer:1986, DaSilveira:1987, Lopez-Ros:1989}}] \label{lem:Lopez-Ros}  Let $(X, g)$ be a complete oriented Riemann surface and let $F : (X, g) \to (\R^3, \delta)$ be an isometric immersion. If this immersion has constant mean curvature and is volume preserving stable, then $F(X)$ is either a plane or a round sphere. 
\end{lemma}

We emphasize that the immersion in this lemma is not required to be proper.

\begin{proposition} [\protect{essentially \cite[Theorem 7]{Ye:1996}}] \label{prop:YeCurvatureEstimates} Let $(M, g)$ be a homogeneously regular Riemannian $3$-manifold. There exists a constant $c>0$ depending only on an absolute bound for the Ricci curvature and a lower bound on the injectivity radius of $(M, g)$ such that every oriented, two-sided, immersed volume preserving stable constant mean curvature surface $\Sigma \subset M$ with $|\mc_\Sigma| \leq 1$ satisfies $\sup_{x \in \Sigma}|h_\Sigma(x)| \leq c$. \end{proposition}

We added the assumption that the mean curvature be bounded to the
statement of \cite[Theorem 7]{Ye:1996} because we have had some
difficulty understanding the details of the sketch of the argument
provided in \cite[Theorem 7]{Ye:1996} otherwise, specifically when the diameter of the surface is large. We have
corresponded with R. Ye about the original statement of Theorem 7 in
[34] and he has kindly shared with us a counterexample to this
original statement. It is not difficult to see that R.~Ye's argument,
cf.\ the proof of Proposition \ref{prop:curvature_estimates_decay} below,
implies the preceding proposition. 

We note that R. Ye's argument also shows that the extrinsic curvature of a volume preserving stable constant mean curvature surface $\Sigma \subset M$ as in the statement of Proposition \ref{prop:YeCurvatureEstimates} is large only when $\Sigma$ is a perturbation of a small geodesic sphere, cf. the proof of \cite[Theorem 18]{Ros:2005}. This complements the estimate in Proposition \ref{prop:curvature_estimates_decay} below. 

Note that typically there exist volume preserving stable constant mean
curvature surfaces  with arbitrarily large mean curvature in $(M, g)$:
They can be constructed as perturbations of small geodesic balls
around non-degenerate critical points for the scalar curvature 
\cite[Theorem 5]{Ye:1996}. If $(M, g)$ is compact, then small
isoperimetric regions will also have this property, cf. \cite[Theorem
2.2]{Johnson-Morgan:2000} (in general dimension) and \cite[Theorem
18]{Ros:2005} (in dimension $3$, by a different argument).

In the following proposition we adapt and slightly refine the argument of R. Ye in \cite{Ye:1996} to derive curvature decay estimates for connected volume preserving stable constant mean curvature surfaces in initial data sets. The result and proof are also similar to the work \cite{Rosenberg-Souam-Toubiana:2010} on strongly stable constant mean curvature surfaces. 

\begin{proposition} \label{prop:curvature_estimates_decay} Let $(M, g)$ be an initial data set and let $K \subset M$ be a non-empty compact subset. There exists a constant $c>0$ depending only on $(M, g)$ and $K$ such that $\sup_{x \in \Sigma} \left( \max \{1, r(x) \}  |h_\Sigma(x)| \right) \leq c$ for every connected volume preserving stable constant mean curvature surface $\Sigma \subset M$ with $|\mc_\Sigma| \leq 1$ and $\Sigma \cap K \neq \emptyset$. 
\begin{proof}
  From Proposition \ref{prop:YeCurvatureEstimates} we see that we may
  focus on points in $\Sigma \setminus B_2$. Suppose the proposition
  is wrong. Then there exists a sequence of volume preserving stable
  constant mean curvature surfaces $\Sigma_k \subset M$ with
  $|\mc_{\Sigma_k}| \leq 1$ such that $\Sigma_k \cap K \neq \emptyset$
  and points $x_k \in \Sigma_k$ such that $\max_{y \in \Sigma_k \cap
    B(x_k, \frac{|x_k|}{2})} (\frac{|x_k|}{2} - |y - x_k| )
  |h_{\Sigma_k} (y)| =: c_k \to \infty$. By Proposition
  \ref{prop:YeCurvatureEstimates} we must have that $|x_k| \to
  \infty$. Let $y_k \in \Sigma_k \cap B(x_k, \frac{|x_k|}{2})$ be a
  point where the maximum is achieved and put $r_k := \frac{|x_k|}{2}
  - |y_k - x_k|$. (This particular `weighted' point picking argument
  for obtaining local curvature estimates on possibly non-compact
  surfaces is as in \cite[p. 389]{Choi-Schoen:1985}.) Rescale the
  asymptotically flat metric $g_{ij}$ on $B(y_k, r_k)$ to the metric
  $\tilde {g}^k_{ij}$ on the ball $B(0, c_k) \subset \R^3$ using the
  transformation $x = y_k + \frac{r_k}{c_k} \tilde x$, i.e. $\tilde
  g^k = \left( \frac{c_k}{r_k}\right)^2 x^*g$. These rescaled metrics
  $\tilde g^k_{ij}$ on $B(0, c_k)$ converge to the Euclidean metric
  $\delta_{ij}$ on $\R^3$ locally uniformly in $\C^2$ by asymptotic
  flatness, since $|x_k| \to \infty$. The rescaled surfaces $\tilde
  \Sigma_k$ are volume preserving stable and have constant mean
  curvature with respect to the rescaled metrics, they pass through
  the origin, and $1= |h_{\tilde \Sigma_k} (0)| \geq \frac{1}{2}
  |h_{\Sigma_k} (\tilde x)|$ for $\tilde x \in B(0, \frac{c_k}{2})$.
  In particular, $\tilde \Sigma_k$ can be written as a Euclidean graph
  with uniform $\C^{2, \alpha}$ norms above a ball of uniform size in
  the tangent plane of every point $\tilde x \in B(0, \frac {c_k}{4} )
  \cap \tilde \Sigma_k$. A standard diagonalization argument shows
  that there is a subsequence (which we neglect to denote separately),
  a complete oriented abstract $\C^{2, \beta}$ Riemannian manifold
  $(\Sigma_\infty, g^\infty)$ with a marked point $p \in
  \Sigma_\infty$, and an isometric immersion $F: \Sigma_\infty \to
  \R^3$ with $F(p) = 0$ and unit normal $\nu \in \Gamma (F^* (T
  \R^3))$ with the following properties: There exist compactly
  supported $\C^{2, \beta}$ functions $u_k: \Sigma_\infty \to \R$
  whose $\C^{2, \beta}$ norms tend to zero uniformly and such that for
  every $R \geq 1$, $\{ F(q) + u_k(q) \nu(q) : q \in
  B_{\Sigma_\infty}(p, R)\} \subset \tilde \Sigma_k$ for every
  sufficiently large $k$ (depending on $R$). It follows that $F:
  \Sigma_\infty \to \R^3$ is a complete immersed volume preserving
  stable constant mean curvature immersion. From Lemma
  \ref{lem:Lopez-Ros} we know that $F(\Sigma_\infty)$ is either a
  plane or a sphere. The first alternative is impossible since by
  construction the length of the second fundamental form of
  $F(\Sigma_\infty)$ at $0 = F(p)$ is $1$. The latter alternative
  would imply that every $\Sigma_k$ contains a `far out' spherical
  component in $B(x_k, \frac{|x_k|}{2})$, contradicting the assumption
  that $\Sigma_k$ is connected and $\Sigma_k \cap K \neq \emptyset$.
\end{proof}
\end{proposition}

We recall the following well-known fact: 

\begin{lemma} \label{lem:comparison_second} Consider on $\R^3 \setminus B(0, 1)$ a metric of the form $g_{ij} = \delta_{ij} + b_{ij}$ where $|x||b_{ij}| + |x|^2 |\partial_k b_{ij}| \leq C'$ for all $x$ such that $|x| \geq 1$. Let $\Sigma  \subset \R^3 \setminus B(0, 1)$ be an oriented surface and let $h_g, h_\delta$ and $\mc_g, \mc_\delta$ denote the $(1, 1)$-second fundamental forms and the mean curvature scalars of $\Sigma$ computed with respect to $g$ and $\delta$. Then $|h_g - h_\delta|_\delta \leq C \frac{|x||h_g|_g + 1}{|x|^2}$ and $|\mc_g - \mc_\delta| \leq C \frac{|x||h_g|_g + 1}{|x|^2}$ for all $|x| \geq r_0$ where $r_0$ and $C$ only depend on $C'$.  
\end{lemma}

We will use the following important result of D. Christodoulou and S.-T. Yau to obtain geometric bounds for volume preserving stable constant mean curvature surfaces. 

\begin {lemma} [\protect{\cite{Christodoulou-Yau:1988}, cf. \cite [Theorem 12]
{Ros:2005}}] \label{lem:Christodoulou-Yau}
Let $(M, g)$ be a Riemannian $3$-manifold  and let $\Sigma \subset M $ be a connected closed volume preserving stable constant mean curvature
surface. Then $\int_\Sigma (\frac{2}{3} \Scal_g + \frac{2}{3} |\ring{ h}|^2 + \mc^2 ) d \H^2_g\leq \frac{64 \pi}{3}$. If $\Sigma$ is a topological sphere, then the bound on the right-hand side can be improved to $16 \pi$. 
\end{lemma}

In Euclidean space, compact volume preserving stable constant mean curvature surfaces are spheres by Lemma \ref{lem:Lopez-Ros}, and the following corollary is obvious. In our more general situation we depend on geometric estimates coming from a bound on the bending energy $\int_\Sigma \mc^2 d \H^2_g$ of a surface from \cite{Simon:1993} and our curvature decay estimates in Proposition \ref{prop:curvature_estimates_decay}.  

\begin{corollary} \label{cor:quadratic_area_growth}
Let $(M, g)$ be an initial data set with non-negative scalar curvature. Given a non-empty compact subset $K \subset M$ there exist constants $C, r_0 \geq 1$ depending only on $(M, g)$ and $K$ with the following property: Let $\Sigma \subset M$ be a connected closed volume preserving stable constant mean curvature surface with $\Sigma \cap K \neq \emptyset$. Let $r_0 \leq r \leq \sigma $ and assume that $\Sigma$ intersects $\partial B_{r}$ transversely. Then 
\begin{eqnarray} \label{eqn:aux_quadratic_growth} 
\frac{\H^2_g(\Sigma \cap (B_\sigma \setminus B_r))}{\sigma^2} \leq C \left( 1 + \frac{\H^1_g (\Sigma \cap \partial B_r)}{r} \right).  
\end{eqnarray} 
If $(M, g)$ has everywhere positive scalar curvature, then there exists a constant $\Theta > 0$ depending only on $(M,g)$ and $K$ such that 
\begin{eqnarray} 
\label{eqn:quadratic_area_growth} \frac{\H^2_g(\Sigma \cap B_\sigma )}{\sigma^2} \leq \Theta \text{ for all } \sigma \geq 1. 
\end{eqnarray}
\begin{proof}
By Lemma \ref{lem:Christodoulou-Yau}, $\int_{\Sigma}
(\frac{2}{3}\Scal_g + \mc^2) d \H^2_g \leq \frac{64\pi}{3}$. Hence
$\int_\Sigma \mc^2 d \H^2_g$ is bounded and we can apply Lemma
\ref{lem:WillmoreAreaBound}, Proposition
\ref{prop:curvature_estimates_decay} (the comment below Proposition \ref{prop:YeCurvatureEstimates} shows that $H$ must stay bounded if $\Sigma$ has diameter, say, greater than one), and Lemma
\ref{lem:WillmoreMonotonicityManifold} (with $\rho = \max_{x\in\Sigma} r(x)$) to establish \eqref{eqn:aux_quadratic_growth}. If in addition the scalar curvature $\Scal_g$ is bounded below by a positive constant near $\bar B_{r_0}$, say on $B_{r_0 + \delta}$, then $\H^2_g (\Sigma \cap B_{r_0 + \delta})$ is bounded above. This implies that there exists $r \in (r_0, r_0 + \delta)$ such that $\Sigma$ intersects $\partial B_r$ transversely and such that $\H^1_g (\Sigma \cap \partial B_r)$ is also a priori bounded, so that \eqref{eqn:quadratic_area_growth} follows from \eqref{eqn:aux_quadratic_growth}. 
\end{proof}
\end{corollary}

The following corollary is standard for minimal surfaces in Euclidean space. In order to obtain the result here, we use the curvature estimates with decay for volume preserving stable minimal surfaces from Proposition \ref{prop:curvature_estimates_decay}.

\begin{corollary} \label{cor:density_at_infinity} Let $(M, g)$ be an
  initial data set. Given a non-empty compact subset $K \subset M$
  there is  a constant $C >0$ depending only on $(M, g)$ and
  $K$ with the following property: For every connected volume preserving stable
  minimal surface $\Sigma \subset M$ with $\Sigma \cap K \neq
  \emptyset$ such that $\H^2_g(\Sigma \cap B_\sigma) \leq \Theta
  \sigma^2$ for all $\sigma \geq 1$ one has that $\sigma \to
  \sigma^{-2}\H^2_\delta (\Sigma \cap B_\sigma) + C
  \Theta \sigma^{-1}$ is monotone increasing for $\sigma \geq
  1$. In particular, the limit of $\sigma^{-2}\H^2_g(\Sigma \cap
  B_\sigma)$ as $\sigma \to \infty$ exists.  
\begin{proof}
We choose $r_0 \geq 1$ so large that the curved and Euclidean Hausdorff measures are equivalent on $M \setminus B_{r_0}$. It follows from Proposition \ref{prop:curvature_estimates_decay} and Lemma \ref{lem:comparison_second} that  $|\mc_\delta| \leq \frac{C}{|x|^2}$ if $r\geq r_0$, choosing $r_0$ even larger if necessary.  There exists $r \in (r_0, r_0 +1)$ so that $\partial B_r$ intersects $\Sigma$ transversely and such that $\H^1_g(\Sigma \cap \partial B_r) \leq \Theta (r_0+1)^2$. An obvious modification of the proof of \cite[(17.3)]{GMT} to surfaces with boundary shows  that for a.e. $\sigma \geq r$ one has that
\begin{eqnarray*}
\frac{d}{d\sigma} \left( \frac{\H_\delta^2 (\Sigma \cap (B_\sigma \setminus B_r))}{\sigma^2} - \int_{\Sigma \cap (B_\sigma \setminus B_r)} \frac{\left|D^{\perp} |x|\right|^2}{|x|^2} d \H^2_\delta \right) &=& \\  - \sigma^{-3} \left( \int_{\Sigma \cap (B_\sigma \setminus B_r)} \mc_\delta (X, \nu) d \H^2_\delta + \int_{\Sigma \cap \partial B_r} (X, \eta) d \H^1_\delta \right)
&\geq& \\   - C \sigma^{-3} \int_{\Sigma \cap (B_\sigma \setminus B_r)} \frac{1}{|x|} d \H^2_\delta                - \frac{r \H_\delta^1(\Sigma \cap \partial B_r)}{\sigma^3} 
&\geq& \\ - \frac{C'  \Theta}{\sigma^2}   - \frac{r \H_\delta^1(\Sigma \cap \partial B_r)}{\sigma^3} 
\end{eqnarray*} 
where $\eta$ is the inward pointing unit normal of $\Sigma \cap \partial B_r$ in $\Sigma$.  Here we have used the quadratic area growth for a bound $\int_{\Sigma \cap (B_\sigma \setminus B_r)} |x|^{-1} d \H^2_\delta \lesssim \Theta \sigma$.
\end{proof}
\end{corollary}

%%%%%%%%%%%%%%%%%%%%%%%%%%%%%%%%%
%%%%%%%%%%%%%%%%%%%%%%%%%%%%%%%%%
%%%%%%%%%%%%%%%%%%%%%%%%%%%%%%%%%

\section{Volume preserving stable CMC surfaces vs. positive scalar curvature} \label{sec:PMT}

The purpose of this section is to carefully establish the following proposition, which extends the scope of a famous argument of R. Schoen and S.-T. Yau in \cite{Schoen-Yau:1979-pmt1} to include our application in the next section:

\begin{proposition} \label{prop:mainPMT} Let $(M, g)$ be an initial data set and let
  $\Sigma \subset M$ be a connected non-compact volume
  preserving stable minimal surface such that $\sup_{\sigma\geq1} \sigma^{-2}
  \H^2_g(\Sigma \cap B_\sigma) < \infty$. Then
  $\int_{\Sigma} (\Scal_g + |h|^2) d \H^2_g \leq 0$.
\end{proposition}

In \cite{Schoen-Yau:1979-pmt1} two different arguments were given to
prove this proposition for the case of area minimizing surfaces
$\Sigma$ (no volume constraint) that lie in a slab of an initial data set. These hypotheses
imply strong stability and quadratic area growth, but they are used in other ways as well in \cite{Schoen-Yau:1979-pmt1}. In Proposition \ref{prop:main} we note that 
non-compact volume preserving stable minimal surfaces with quadratic area growth are strongly stable. A related observation is that non-compact
minimal ``locally isoperimetric" boundaries are area minimizing,
cf. \cite[Appendix C]{Eichmair-Metzger:2010}. Another difference
from the original argument in \cite{Schoen-Yau:1979-pmt1} here is how
we obtain that $\int_\Sigma \kappa d \H^2_g \leq 0$, cf. Proposition
\ref{prop:PMT}. The reader should also compare our results here  with Theorem 3 in
\cite{Fischer-Colbrie-Schoen:1980} and its proof. 

Proposition \ref{prop:mainPMT} will follow by combining Proposition \ref{prop:main}, Corollary \ref{cor:cohn-vossen}, and Proposition \ref{prop:PMT} below. 

\begin{lemma} Let $\Sigma \subset M$ be a non-compact surface with bounded mean curvature and such that $\H^2_g(\Sigma \cap B_\sigma) \leq \Theta \sigma^2$ for all $\sigma \geq 1$. For every $\epsilon >0$ there exists a Lipschitz function $\chi_\epsilon$ defined on $\Sigma$ such that (i) $\chi_\epsilon$ has compact support and $\spt (\chi_\epsilon) \cap B_{\epsilon^{-1}} = \emptyset$, (ii) $\int_\Sigma |\bar \nabla \chi_\epsilon|^2 d \H^2_g \leq \epsilon$, and such that (iii) $0 \leq \chi_\epsilon \leq 1$ and $\int_\Sigma \chi_\epsilon d \H^2_g = 1$. 
\begin{proof} One can use that $\Sigma$ has at most quadratic area growth and that the area of $\Sigma$ is infinite (which follows from the monotonicity formula and the assumption that $\Sigma$ is non-compact) to construct a non-negative `hat' function on $\Sigma$ that first increases logarithmically from $0$ to $1$, then stays equal to $1$ to pick up integral, and then decays logarithmically to $0$. A computation -- ``the logarithmic cut-off trick" applied precisely as in \cite[bottom of p. 52, p. 54]{Schoen-Yau:1979-pmt1} -- shows that one can construct such functions arbitrarily far out and with arbitrarily small Dirichlet energy by taking enough space to increase and decrease. One can then scale the function down to take values between $0$ and $1$ and so that its integral becomes equal to~$1$.    \end{proof}
\end{lemma}

Recall that complete non-compact volume preserving stable constant mean curvature
surfaces are strongly stable outside a compact set (cf.\ \cite[Section
2]{Koh:1987} and \cite[Lemma 4]{Cheung:1991}). The proof of Proposition \ref{prop:mainPMT} depends on the following, stronger conclusion:

\begin{proposition} \label{prop:main} Let $\Sigma \subset M$ be a
  non-compact volume preserving stable constant mean
  curvature surface such that $\H^2_g(\Sigma \cap B_\sigma) \leq
  \Theta \sigma^2$ for all $\sigma \geq 1$. Then $\Sigma$ is strongly
  stable and $\int_\Sigma (|h|^2 + \Ric(\nu, \nu)) d \H^2_g \leq 0$.
  \begin{proof} We will use here that $\int_{\Sigma} |\Ric| d \H^2_g <
    \infty$, which follows from the quadratic area growth and the
    decay $|\Ric| = O(|x|^{-3})$, and is proven in
    \cite[p. 52-53]{Schoen-Yau:1979-pmt1}. Fix $u \in \C^1_c(\Sigma)$
    and let $\epsilon >0$ be such that $\spt(u) \subset
    B_{\epsilon^{-1}}$. Let $\alpha:= \int_\Sigma u d \H^2_g$. Then
    $u_\epsilon := u - \alpha \chi_\epsilon$ is Lipschitz with compact
    support and has mean zero. Hence $\int_\Sigma |\bar
    \nabla u_\epsilon|^2 d \H^2_g \geq \int_\Sigma (|h|^2 + \Ric (\nu,
    \nu)) u_\epsilon^2 d \H^2_g$. Note that
\begin{eqnarray*} \int_\Sigma |\bar \nabla u_\epsilon|^2 d \H^2_g = \int_\Sigma |\bar \nabla u|^2 + \alpha ^2 |\bar \nabla \chi_\epsilon|^2 d \H^2_g = \int_\Sigma |\bar \nabla u|^2 d \H^2_g + O(\epsilon), \end{eqnarray*}
and that  
\begin{eqnarray*} && \int_\Sigma (|h|^2 + \Ric(\nu, \nu)) u_\epsilon^2 d \H^2_g\\ &=&  \int_\Sigma (|h|^2 + \Ric(\nu, \nu)) u^2 d \H^2_g + \alpha^2 \int_\Sigma (|h|^2 + \Ric(\nu, \nu)) \chi_\epsilon^2 d \H^2_g \\ &\geq& \int_\Sigma (|h|^2 + \Ric(\nu, \nu)) u^2 d \H^2_g - \alpha^2 \int_{\Sigma \setminus B_{\epsilon^{-1}}} |\Ric| d \H^2_g \\ &\geq& \int_\Sigma (|h|^2 + \Ric(\nu, \nu)) u^2 d \H^2_g - O(\epsilon). 
\end{eqnarray*}
Putting these inequalities together and letting $\epsilon \searrow 0$
we see that $$\int_\Sigma |\bar \nabla u|^2 d \H^2_g \geq \int_\Sigma
(|h|^2 + \Ric(\nu, \nu))u^2 d \H^2_g.$$ So $\Sigma$ is indeed strongly
stable. That $$\int_\Sigma (|h|^2 + \Ric(\nu, \nu)) d \H^2_g \leq 0$$
now follows from the logarithmic cut-off trick exactly as in \cite[top
of p. 55]{Schoen-Yau:1979-pmt1}.
\end{proof}
\end{proposition}

\begin{corollary}
  [\protect{\cite{Schoen-Yau:1979-pmt1}}] \label{cor:cohn-vossen} Let
  $(M, g)$ be an initial data set, and let $\Sigma \subset M$ be a
  non-compact volume preserving stable constant mean
  curvature surface such that $\H^2_g(\Sigma \cap B_\sigma) \leq
  \Theta \sigma^2$ for all $\sigma \geq 1$. Then $\int_\Sigma |h|^2 d
  \H^2_g < \infty$, $\Sigma$ is minimal, and $\int_\Sigma |\kappa| d
  \H^2_g < \infty$. If, in addition, $\Sigma$ is connected and
  $\int_\Sigma (\Scal_g + |h|^2) d \H^2_g > 0$, then $\Sigma$ is
  conformally diffeomorphic to $\mathbb{C}$.
  \begin{proof} This follows exactly as in
    \cite[p. 54-55]{Schoen-Yau:1979-pmt1} and uses the fact that $\int_\Sigma
    |\Ric| d \H^2_g < \infty$, that $\int_\Sigma (|h|^2 + \Ric(\nu,
    \nu)) d \H^2_g \leq 0$ from Proposition \ref{prop:main} (implying
    that $\int_\Sigma |h|^2 \H^2_g < \infty$ and hence that the mean
    curvature of $\Sigma$ vanishes, since $\Sigma$ has infinite area), the Gauss equation on minimal
    surfaces in the form $ 2 (|h|^2 + \Ric(\nu, \nu)) = \Scal_g - 2
    \kappa + |h|^2$, and the Cohn-Vossen inequality.
\end{proof}
\end{corollary}

In a fixed coordinate system for the asymptotically flat end of an initial data set $(M, g)$ we can consider planes $\Pi := \{ (x_1, x_2, x_3) : a_1 x_1 + a_2 x_2 + a_3 x_3 = 0\}$ where $(a_1, a_2, a_3) \neq 0$. 

\begin{lemma} \label{lem:annular_pieces} Let $\Sigma \subset M$ be a connected non-compact strongly stable minimal surface such that $\H^2_g(\Sigma \cap B_\sigma) \leq \Theta \sigma^2$ for all $\sigma\geq1$. For every sequence of radii $\sigma_i \to \infty$ there exists a subsequence $\sigma_{i'} \to \infty$ and a plane $\Pi$ such that the intersection of $\Sigma$ with the normal cylinder above each annulus $\Pi \cap (B_{3 \sigma_{i'}} \setminus B_{\sigma_{i'}})$ is a union of finitely many disjoint graphs above this annulus. The scale invariant norms of the defining functions of these graphs tend to zero as $\sigma_{i'} \to \infty$. 
\begin{proof} 
The gist of our argument here is similar to W. Fleming's proof \cite{Fleming:1962} of the Bernstein theorem in $\R^3$. Consider the rescaled surfaces $\Sigma_i := \sigma_i^{-1} \Sigma$ in $\R^3 \setminus B(0, \sigma_i^{-1})$. Using the curvature estimates in Proposition \ref{prop:curvature_estimates_decay}, the quadratic area growth of $\Sigma_i$ that is inherited from $\Sigma$, and a diagonal subsequence argument we conclude that there exists a subsequence $\sigma_{i'}$ so that $\Sigma_{i'}$ converges locally smoothly to a strongly stable minimal surface $\Sigma_\infty$ in $\R^3 \setminus \{0\}$, possibly with multiplicity. It follows from Corollary \ref{cor:density_at_infinity} that $\sigma^{-2}\H^2_\delta(\Sigma_\infty \cap B_\sigma)$ is constant in $\sigma>0$ (it equals $\lim_{\sigma \to \infty} \sigma^{-2}\H_g^2(\Sigma \cap B_\sigma)$). The monotonicity formula applied exactly as in \cite[Theorem 19.3]{GMT} shows that $\Sigma_\infty$ is a cone. In $\R^3$ this means that the support of $\Sigma_\infty$ is a union of planes. Since $\Sigma$ is embedded, there is exactly one plane, possibly assumed with multiplicity $\geq 1$. 
\end{proof}
\end{lemma}

\begin{proposition} \label{prop:PMT} Let $(M, g)$ be an initial data set. Let $\Sigma \subset M$ be a connected non-compact strongly stable minimal surface such that $\H^2_g(\Sigma \cap B_\sigma) \leq \Theta \sigma^2$ for all $\sigma \geq 1$. Then $\int_\Sigma (\Scal_g + |h|^2) d \H^2_g \leq 0$. 
\begin{proof} The fundamental idea here is that of the ``\emph{Second
    Proof}" in \cite[p. 57-63]{Schoen-Yau:1979-pmt1}. Our argument for
  showing that the total geodesic curvature of certain circles in
  $\Sigma$ approaches $2\pi$ is different and more general. 
  
  Assume that $\int_\Sigma (\Scal_g + |h|^2) d \H^2_g > 0$. Then we
  know from Corollary \ref{cor:cohn-vossen} that $\Sigma \cong
  \mathbb{C}$. Consider a sequence $\sigma_i' \to \infty$ as in Lemma
  \ref{lem:annular_pieces}. For ease of notation, assume that $\Pi$ is the $x_1x_2$-plane in the asymptotic
  coordinate system. For $\sigma \geq 1$ we denote $C_\sigma = \{
  (x_1, x_2, x_3) : x_1^2 + x^2_2 \leq \sigma^2\}$. We know that
  $\Sigma \cap \partial C_{2 \sigma_i'}$ consists of a union of
  disjoint circles that are all graphical above $\Pi$.  The number of
  these circles equals the multiplicity $m$ of $\Pi$ in the blow down
  limit. The argument of \cite[bottom of p. 57]{Schoen-Yau:1979-pmt1}
  shows that each of these circles bounds a disk in $\Sigma \cap C_{2
    \sigma_i'}$.  (One uses that the boundaries of $C_\sigma$ are mean
  convex for $\sigma$ large to see that the bounded components of
  $\Sigma \setminus \partial C_{2 \sigma_{i'}}$ lie in $C_{2 \sigma_{i'}}$, and
  that $\Sigma \cong \mathbb{C}$.)  By Lemma \ref{lem:annular_pieces},
  these circles converge to round congruent circles in $m \Pi$ upon
  blow down. Hence, by scale invariance, the total geodesic curvature
  of each of the circles in $\Sigma \cap C_{2 \sigma_i'}$ approaches
  $2\pi$ as $i' \to \infty$. By the Gauss-Bonnet theorem, this implies
  that $\int_{\Sigma \cap C_{2 \sigma_{i'}}} \kappa d \H^2_g =
  o(1)$. Since $\int_\Sigma |\kappa| d\H^2_g < \infty$ we obtain
  that $\int_\Sigma \kappa d \H_g^2 = 0$. Together with Proposition \ref{prop:main} this leads to the
  contradiction $$0 < \int_\Sigma (\Scal_g - 2 \kappa + |h|^2) d
  \H^2_g = 2 \int_\Sigma (|h|^2 + \Ric(\nu, \nu)) d \H^2_g \leq 0.$$

\end{proof}
\end{proposition}

%%%%%%%%%%%%%%%%%%%%%%%%%%%%%%%%%
%%%%%%%%%%%%%%%%%%%%%%%%%%%%%%%%%
%%%%%%%%%%%%%%%%%%%%%%%%%%%%%%%%%

\section{Proof of Theorems \ref{thm:volume_stable_center} and \ref{thm:volume_stable_center_positive}} \label{sec:main_theorem_proof}

\begin{proof} [Proof of Theorem \ref{thm:volume_stable_center}] We proceed by contradiction: Assume that there exists a
  sequence $\{\Sigma_i\}$ of connected closed volume
  preserving stable constant mean curvature surfaces such that
  $\H^2_g(\Sigma_i \cap B_\sigma) \leq \Theta \sigma^2$ for all
  $\sigma \geq 1$ and all $i = 1, 2, \ldots$, such that $\Sigma_i \cap
  K \neq \emptyset$, and such that $\H^2_g (\Sigma_i) \geq i$. These
  assumptions imply that $\max_{x\in\Sigma} |x|$ tends to infinity. From Lemma \ref{lem:Christodoulou-Yau} we know that $\H^2_g(\Sigma_i) \mc_{\Sigma_i}^2 \leq \int_{\Sigma_i} (\frac{2}{3} \Scal_g + \mc_{\Sigma_i}^2) d \H^2_g \leq \frac{64\pi}{3}$ so that $\mc_{\Sigma_i} \to 0$ as $i \to \infty$. Using the curvature estimates from Theorem \ref{prop:curvature_estimates_decay} we can pass $\Sigma_i$ to a subsequential limit $\Sigma_\infty \subset M$ (where the convergence is as in the proof of Theorem \ref{prop:curvature_estimates_decay}). The components of $\Sigma_\infty$ are unbounded. Let $\hat \Sigma_\infty$ be a component of $\Sigma_\infty$ such that $\Sigma_\infty \cap K \neq \emptyset$. Note that $\hat \Sigma_\infty$ might be assumed with multiplicity $>1$. It is easy to see that $\hat \Sigma_\infty$ is a complete non-compact embedded orientable two-sided volume preserving stable minimal surface with quadratic area growth. By Proposition \ref{prop:mainPMT}, $\int_{\hat \Sigma_\infty} (\Scal_g + |h|^2) d \H^2_g = 0$. This contradicts our assumption that the scalar curvature $\Scal_g$ is strictly positive in a neighborhood of $K$.  
\end{proof}

\begin{proof}  [Proof of Theorem \ref{thm:volume_stable_center_positive}]
Use Corollary \ref{cor:quadratic_area_growth} to show that Theorem \ref{thm:volume_stable_center} applies. 
\end{proof}

%%%%%%%%%%%%%%%%%%%%%%%%%%%%%%%%%
%%%%%%%%%%%%%%%%%%%%%%%%%%%%%%%%%
%%%%%%%%%%%%%%%%%%%%%%%%%%%%%%%%%

\section {On the number of components of volume preserving stable CMC surfaces} \label{sec:number_of_components}

In this section we collect results on the number of components of
isoperimetric surfaces and ``large" volume preserving stable constant mean curvature surfaces
in initial data sets.

\begin{proposition} \label{prop: mc_bdd}
Let $(M, g)$ be a Riemannian $3$-manifold with non-negative scalar curvature and let $\Sigma \subset M$ be a closed volume preserving stable constant mean curvature surface -- not necessarily connected -- with
constant mean curvature $\mc_{\Sigma}$. Then $$ \mc^2_\Sigma \leq \max \left( -
2 \inf_{x \in \Sigma} \Ric(\nu, \nu), \frac{64 \pi}{3}
\H^2_g(\Sigma)^{-1} \right).$$
\begin{proof}
By a standard stability
argument (cf. \cite[p. 294]{Ritore-Ros:1992}, \cite[p. 73]{Bray:1998}), if $0 <
| h|^2 + \Ric(\nu, \nu)$ along $\Sigma$, then $\Sigma$ is connected. By Lemma \ref{lem:Christodoulou-Yau} we then
have that $ \H_g^2(\Sigma) \mc_\Sigma^2  \leq \frac{64
\pi}{3}$. Note that $\mc_\Sigma^2 \leq 2 (|h|^2 + \Ric(\nu, \nu)) - 2 \Ric (\nu, \nu)$. 
\end{proof}
\end{proposition}

It is interesting to compare the estimate in Proposition \ref{prop: mc_bdd} with the main result in \cite{metzger:2007-hmax}, where an effective bound on $\mc_\Sigma$ is derived for constant mean curvature surfaces $\Sigma \subset M$  that enclose a component of the horizon. 

\begin{proposition} \label{prop: Willmore energy bounded below}
Let $\Sigma$ be a closed surface in an initial data set $(M, g)$. Then $\int_{\Sigma} |h|^2 d \H^2_g \geq 8 \pi - o(1)$ as $r_{\min} := \inf \{ r(x) : x \in \Sigma \} \to \infty$.  
\begin{proof} We may assume that $\int_\Sigma \mc^2 d \H^2_g \leq 16 \pi$ (otherwise we are done). We will compare geometric quantities of $\Sigma$ in $(M, g)$ with those computed with respect to the Euclidean metric in the asymptotically flat coordinate chart. We will denote the former with a subscript $g$ and the latter with a subscript $\delta$, for emphasis. First,  $\int_{\Sigma} |h_\delta|_\delta^2 d \H^2_\delta \geq \frac{1}{2} \int_{\Sigma} \mc_\delta^2 d \H^2_\delta \geq 8 \pi$ (see \cite[(16.32)]{Gilbarg-Trudinger:1998} for a derivation of the second inequality). By Lemma  \ref{lem:comparison_second}, $|h_g - h_\delta|_\delta \leq C' \left(  \frac{|h_\delta|_\delta}{r} + \frac{1}{r^2} \right)$ for large $r$. Using the estimates in Appendix \ref{sec:integral_decay_estimates}, we compute that
  \begin{equation*}
    \begin{split}
      &(1 + o(1)) \int_\Sigma |h_g|^2_g d \H^2 _g
      \\
      &\geq \int_{\Sigma}
      |h_g|^2_\delta d \H^2_\delta 
      \\
      &= \int_{\Sigma} |h_\delta|^2_\delta d \H^2_\delta +
      \int_{\Sigma} (|h_g|_\delta - |h_\delta|_\delta)( |h_g|_\delta +
      |h_\delta|_\delta) d\H^2_\delta 
      \\
      &\geq \int_{\Sigma} |h_\delta|^2_\delta d \H^2_\delta -  2 C'^2
      \int_\Sigma \left( \frac{|h_\delta|_\delta}{r} + \frac{1}{r^2}
      \right) \left( |h_\delta|_\delta + \frac{1}{r^2} \right) d
     \H^2_\delta 
     \\
     &\geq (1 - o(1))\int _{\Sigma} |h _\delta|^2_\delta d \H^2_\delta - o(1) \text{ as } r_{min} \to \infty. 
    \end{split}
  \end{equation*}
\end{proof} 
\end{proposition}

\begin{proposition} \label{prop:at_most_one_component_outside} Let $(M, g)$ be an initial data set. There is a constant $r_0 \geq 1$ so that every closed volume preserving stable constant mean curvature surface $\Sigma \subset M$ contains at most one component $\Sigma'$ with $\Sigma' \cap B_{r_0} = \emptyset$.

\begin{proof} Assume that there are two components $\Sigma'$ and $\Sigma''$ of $\Sigma$ that are both disjoint from $B_{r_0}$. Assume that $\H^2_g(\Sigma'') \geq \H^2_g(\Sigma')$. As in the proof of \cite[Proposition 5.3]{Huisken-Yau:1996} one can combine Lemma \ref{lem:Christodoulou-Yau} and Lemma \ref{lem:surface_integral_decay} (using that $|\Ric|(x) = O(r^{-3})$) to conclude that $\int_{\Sigma' \cup \Sigma''} \mc_\Sigma^2 d \H^2_g$ is bounded.  The function that equals $\H^2_g(\Sigma'')$ on $\Sigma'$ and $- \H^2_g(\Sigma')$ on $\Sigma''$ generates a volume preserving normal deformation of $\Sigma$. Hence 
\begin{eqnarray*} 
0 &\geq& \int_{\Sigma'} \left (|h|^2 + \Ric(\nu, \nu) \right) d \H^2 _g  + \frac{\H^2_g(\Sigma')^2}{\H^2_g(\Sigma'')^2} \int_{\Sigma''}  \left (|h|^2 + \Ric(\nu, \nu) \right) d \H^2_g \nonumber \\
&\geq& \int_{\Sigma'} |h|^2 d \H^2_g - \int_{\Sigma' \cup \Sigma''} |\Ric| d \H^2_g \nonumber  \\
&\geq & \int_{\Sigma'} |h|^2 d \H^2_g - O(r_0^{-1})
\end{eqnarray*}
as $r_{min} (\Sigma' \cup \Sigma '') \geq r_0 \to \infty$, where we applied Lemma \ref{lem:surface_integral_decay} again. For $r_0$ large enough, this is in contradiction with the result in Proposition \ref{prop: Willmore energy bounded below}. 
\end{proof}
\end{proposition}

\begin{proposition} \label{prop:number_components_bounded} Let $\Sigma$ be a closed volume preserving stable constant mean curvature surface in an initial data set $(M, g)$. Then the number of components $n_\Sigma$ of $\Sigma$ is bounded in terms of $\H^2_g(\Sigma \cap B_{2 r_0})$ where $r_0$ is as in Proposition \ref{prop:at_most_one_component_outside}. 
\begin{proof} From the proof of Proposition \ref{prop: mc_bdd} we know that if the mean curvature of $\Sigma$ is larger than a constant depending only on $(M, g)$, then $\Sigma$ is connected. On the other hand, given an upper bound on the mean curvature of $\Sigma$, the monotonicity formula shows that  every component of $\Sigma \cap B_{2 r_0}$ which intersects $B_{r_0}$ makes a definite contribution to $\H^2_g(\Sigma \cap B_{2 r_0})$. Hence the number of such components is bounded in terms of $\H^2_g(\Sigma \cap B_{2 r_0})$. By Proposition \ref{prop:at_most_one_component_outside}, $\Sigma$ has at most one component that is disjoint from $B_{r_0}$.  
\end{proof}
\end{proposition}

We record the following well-known isoperimetric inequality which follows in a standard way from the Sobolev inequality, cf. \cite[Lemma 2.4]{Eichmair-Metzger:2010}.  %Check reference with published version.

\begin{lemma} \label{lem: crude isoperimetric inequality}
Let $(\hat M, \hat g)$ be a complete Riemannian manifold diffeomorphic to $\R^3$ that contains the initial data set $(M, g)$ isometrically. There exists a constant $\gamma >0$ depending only on $(\hat M, \hat g)$ such that
$\CL_{\hat g}^3(U)^{\frac{2}{3}} \leq \gamma \H_{\hat g}^2 (\partial^* U)$ holds for every bounded Borel set $U \subset \hat M$ with finite perimeter.
\end{lemma}

\begin{corollary} \label{cor:number_of_components_bounded} Let $\Omega$ be an isoperimetric region (i.e. a minimizer in \eqref{eqn:isoperimetric_profile}) in an initial data set $(M, g)$. Then the number of components $n_\Sigma$ of $\Sigma := \partial \Omega$ is bounded by a constant depending only on $(M, g)$. If $(M, g)$ has non-negative scalar curvature, then $\mc_{\Sigma}^2 \H^2_g(\Sigma) \leq  \frac{64 \pi}{3} n_\Sigma$. In particular, the mean curvature of the boundaries of isoperimetric regions tends to zero with their volume.  
\begin{proof}
If $r_0$ is large, depending only on $(M, g)$, and if $\CL^3_g(\Omega \cap B_{2 r_0})$ is small compared to $\CL^3_g (B_{2 r_0})$, then consider the region obtained from replacing the part of $\Omega$ that lies in $B_{2 r_0}$ by a coordinate ball of $g$-volume $\CL^3_g (\Omega \cap B_{2 r_0})$ near the boundary of $B_{2 r_0}$, and use this region as a competitor for least area under the volume constraint to obtain an explicit estimate for $\H^2_g (\partial \Omega \cap B_{2r_0})$. If $\CL^3_g(\Omega \cap B_{2 r_0})$ is not small when compared to $\CL^3_g (B_{2 r_0})$, we replace $\Omega \cap B_{2r_0}$ by a centered coordinate ball $B_{r'}$ of volume $\CL^3_g(\Omega \cap B_{2r_0})$. It follows that $\H^2_g(\partial \Omega \cap B_{2r_0})$ is bounded explicitly in terms of $r_0$. Taking $r_0$ even larger if necessary we may apply Proposition \ref{prop:number_components_bounded} to obtain a bound $n_\Sigma$ on the number of components of $\Sigma = \partial \Omega$. That $\mc_{\Sigma}^2 \H^2_g(\Sigma) \leq  \frac{64 \pi}{3} n_\Sigma$ then follows from Lemma \ref{lem:Christodoulou-Yau}. The isoperimetric inequality in Lemma \ref{lem: crude isoperimetric inequality} shows that $\H^2_g(\partial \Omega) \to \infty$ as $\CL^3_g(\Omega) \to \infty$, proving the last claim. 
\end{proof}
\end{corollary}

%%%%%%%%%%%%%%%%%%%%%%%%%%%%%%%%%
%%%%%%%%%%%%%%%%%%%%%%%%%%%%%%%%%
%%%%%%%%%%%%%%%%%%%%%%%%%%%%%%%%%
%%%%%%%%%%%%%%%%%%%%%%%%%%%%%%%%%
%%%%%%%%%%%%%%%%%%%%%%%%%%%%%%%%%
%%%%%%%%%%%%%%%%%%%%%%%%%%%%%%%%%

\appendix 

\section{Integral decay estimates} \label{sec:integral_decay_estimates}

In this appendix we collect estimates for surface integrals of decaying quantities. Our computations take place on the part of an initial data set $(M, g)$ that is diffeomorphic to $\{|x| \geq 1\} \subset \R^3$ where \begin{eqnarray} \label{eqn: decay assumptions integral decay} r|g_{ij } - \delta_{ij}| + r^2|\partial _k g_{ij}|  \leq C \text{ for all } r:= |x| \geq 1. \end{eqnarray} 

\begin{lemma}  
 [\protect{\cite[Lemma 5.2]{Huisken-Yau:1996}}]\label{lem:surface_integral_decay}
  Let $(M,g)$ be an initial data set for which the decay assumptions \eqref{eqn: decay assumptions integral decay} hold. For every exponent $p>2$ there exist constants $C_1$ and $r_0 \geq 1$ such that for every closed surface $\Sigma \subset M$ with $\Sigma \cap B_{\rho} = \emptyset$ for some  $\rho \geq r_0$ one has the estimate
  \begin{equation*}
    \int_\Sigma r^{-p} d \H^2_g \leq C_1 \rho ^{2-p} \int_\Sigma \mc^2 d \H^2_g.
  \end{equation*}
\end{lemma}

We need the following extension of the previous lemma. The proof is a slight modification of the proof in \cite{Huisken-Yau:1996}:

\begin{lemma}
  \label{lem: surface_integral_decay improved}
  Let $(M,g)$ be an initial data set for which the decay assumptions \eqref{eqn: decay assumptions integral decay} hold. For  every exponent $p>2$ there exist constants $C_2$ and
  $r_0  \geq 1$ such that for every $\rho \geq r_0$ and every surface $\Sigma \subset M \setminus B_\rho$ with $\partial \Sigma \subset \partial B_{\rho}$  one has the estimate
  \begin{equation*}
    \int_{\Sigma \setminus B_{\rho}} r^{-p} d \H^2_g \leq C_2 \rho^{2-p} \left( \int_{ \Sigma \setminus B_{\rho}} H^2 d \H^2_g + \frac{ \H^1_g(\Sigma \cap \partial B_{\rho})}{\rho} \right).
  \end{equation*}
\begin{proof}
  Let $\partial_r$ be the radial vector field
  $\sum_{i=1}^3\frac{x_i}{\sqrt{x_1^2 + x_2^2 + x_3^2}} \partial_i$ in
  the asymptotically flat end. Note that for every $p \in \R$ one has
  that $\div_\Sigma \left( r^{1- p} \partial _r \right) = (2 - p)r^{-
    p} + p r^{- p} g (\nu, \partial_r)^2 + O(r^{- p -
    1})$. Integration by parts on the surface gives
  that
  \begin{multline*}
    \int_{\Sigma \setminus B_\rho} \div_\Sigma ( r^{1-
      p} \partial _r ) d \H^2_g
    \\
    = \int_{\Sigma \setminus B_\rho} H
    r^{1-p}g ( \partial _r, \nu) d \H^2_g + \int_{\Sigma \cap \partial
      B_{\rho}} r^{1-p} g(\partial _r, \eta) d \H^1_g
  \end{multline*}
  where $\eta$ is
  the co-normal of the boundary of $\Sigma \setminus
  B_{\rho}$. Specializing to $p = 2$ this implies that
  \begin{multline*}
    \int_{\Sigma \setminus B_\rho} \frac{g(\nu, \partial_r)^2}{r^2} d
    \H^2_g
    \\
    \leq C_2 \left(  \int_{ \Sigma \setminus B_{\rho}} H^2 d
      \H^2_g + \frac{ \H^1_g(\Sigma \cap \partial B_{\rho})}{\rho}
    \right) + O\left( \int_{\Sigma \setminus B_{\rho}} r^{-3} d
      \H^2_g\right).
  \end{multline*}
  Combining this with the estimate obtained for $p = 3$ one obtains that 
$$ \int_{\Sigma \setminus B_{\rho}} \frac{g(\nu, \partial_r)^2}{r^2} d \H^2_g \leq C_2 \left(  \int_{ \Sigma \setminus B_{\rho}} H^2 d \H^2_g + \frac{ \H^1_g(\Sigma \cap \partial B_{\rho})}{\rho} \right).$$ The estimate for general exponents now follows easily from this. 
\end{proof}
\end{lemma}

%%%%%%%%%%%%%% %%%%%%%%%%%%%%%%%%%
%%%%%%%%%%%%%%%%%%%%%%%%%%%%%%%%%
%%%%%%%%%%%%%%%%%%%%%%%%%%%%%%%%%

\section{Bending energy and area growth} \label{sec:Willmore_Appendix}

Here we modify the proofs of some results in \cite{Simon:1993} so we can apply them in a manifold setting:

\begin{lemma} \label{lem:WillmoreAreaBound} Let $(M, g)$ be an initial
  data set. There exists $r_0\geq1$ depending only on $(M, g)$ such
  that for all $\rho \geq r \geq r_0$ and every bounded surface
  $\Sigma \subset B_\rho \setminus B_r$ with $\partial \Sigma \subset \partial B_r$ one has that  $\H^2_g(\Sigma) \leq 4 \rho^2 \int_{\Sigma} \mc^2 d \H^2_g + 4 r \H_g^1(\partial \Sigma)$.

\begin{proof}  We adapt the argument in \cite[bottom of p. 285]{Simon:1993}. Let $X := \sum_{i=1}^3 x_i \partial_i$ be the position vector field, and choose $r_0$ so large that $g(X, X)|_x \leq 4 |x|^2$ and that the $g$-trace of $\nabla_g X$ over any $2$-dimensional subspace of $T_x M$ is at least $1$ provided that $|x| \geq r_0$. Then for $r \geq r_0$ one has that
\begin{eqnarray*}
\H_g^2(\Sigma) \leq \int_{\Sigma} \div_\Sigma(X) d \H^2_g = \int_{\Sigma} \mc (X, \nu) d \H^2_g + \int_{\partial \Sigma} (X, \eta) d \H^1_g
\end{eqnarray*} where $\eta$ is the outward pointing normal of $\partial \Sigma$. Hence $$\H^2_g(\Sigma) \leq 2 \rho \H^2_g(\Sigma)^{\frac{1}{2}} \left( \int_{\Sigma } \mc^2 d \H^2_g \right)^{\frac{1}{2}} + 2 r \H^1_g(\partial \Sigma)$$ and the claim follows. 

\end{proof}
\end{lemma}

\begin{lemma} Let $r > 0$ and let $\Sigma \subset \R^3 \setminus B_r$ be a surface with $\partial \Sigma \subset \partial B_r$. There exists a universal constant $C_3$ such that for $r \leq \sigma < \rho < \infty$ one has that 
\begin{eqnarray*}
\frac{\H^2_\delta (\Sigma \cap B_\sigma )}{ \sigma^2} \leq C_3 \left( \frac{\H^2_\delta (\Sigma \cap B_\rho)}{ \rho^2} + \int_{\Sigma \cap B_\rho } \mc_\delta^2 d \H^2_\delta + \frac{\H^1_\delta (\partial \Sigma)}{\sigma}\right)
\end{eqnarray*}
\begin{proof}
This is a simple extension of the proof of inequality $(1.3)$ in \cite{Simon:1993} to the case where the surface has an \emph{inner} boundary. 
\end{proof}
\end{lemma}

\begin{lemma} \label{lem:WillmoreMonotonicityManifold} Let $(M, g)$ be an initial data set. There exists a constant $r_0 \geq 1$ such that the following holds: For $r_0 \leq \sigma < \rho$ and every bounded surface $\Sigma \subset M \setminus B_{\sigma}$ with $\partial \Sigma \subset \partial B_\sigma$ and such that $\sup_{x \in \Sigma}|h_\Sigma(x)| |x| \leq c$ one has that 
\begin{eqnarray*} 
\frac{\H^2_g (\Sigma \cap B_\sigma)}{ \sigma^2} \leq C_4 \left( \frac{\H^2_g (\Sigma \cap B_\rho)}{ \rho^2} + \int_{\Sigma} \mc_g^2 d \H^2_g + \frac{\H^1_g (\partial \Sigma)}{\sigma} \right)
\end{eqnarray*}
where $C_4$ depends only on $(M, g)$ and $c$. 
\begin{proof}
For sufficiently large $r_0$ the curved Hausdorff measures are equivalent to the Euclidean ones on $M \setminus B_{r_0}$. We know that $|\mc_\delta - \mc_g| \lesssim  \frac{|x| |h_\Sigma| +1  }{|x|^2}$ so that 
\begin{eqnarray*} \int_{\Sigma \cap B_\rho } \mc_\delta^2 d \H^2 _\delta \sim \int_{\Sigma \cap B_\rho} \mc_\delta^2 d \H^2 _g  \lesssim   \int_{\Sigma} \mc_g^2 d \H^2 _g + \int_{\Sigma} |x|^{-4} d \H^2 _g. 
\end{eqnarray*} 
Moreover, $ \int_{\Sigma} |x|^{-4} d \H^2 _g \lesssim \sigma^{-2} \left( \int_{\Sigma} \mc_g^2 d \H^2_g + \frac{\H^1_g(\partial \Sigma)}{\sigma}\right)$ by Lemma \ref{lem: surface_integral_decay improved}. 
\end{proof}
\end{lemma}

%%%%%%%%%%%%%%%%%%%%%%%%%%%%%%%%%
%%%%%%%%%%%%%%%%%%%%%%%%%%%%%%%%%
%%%%%%%%%%%%%%%%%%%%%%%%%%%%%%%%%

\bibliographystyle{amsplain}
\bibliography{references}

\end{document}